\documentclass[12pt]{amsart}

\usepackage[a4paper,margin=2.3cm]{geometry}
\usepackage{amsmath,amssymb,amsthm,mathtools}
\usepackage{enumitem}
\usepackage{hyperref}
\hypersetup{hidelinks}

\newtheorem{theorem}{Theorem}[section]
\newtheorem{proposition}[theorem]{Proposition}
\newtheorem{lemma}[theorem]{Lemma}

\newtheorem{conjecture}[theorem]{Conjecture}
\theoremstyle{definition}

\theoremstyle{remark}
\newtheorem{remark}[theorem]{Remark}
\numberwithin{equation}{section}

\DeclareMathOperator{\spr}{\rho}

\newcommand{\R}{\mathbb R}
\newcommand{\N}{\mathbb N}
\newcommand{\Z}{\mathbb Z}
\newcommand{\A}{\mathcal A}

\newcommand{\norm}[1]{\left\lVert #1\right\rVert}

\title[Schauder-Basis
Properties of the Daubechies Wavelet Packets]{On a Conjecture about Schauder-Basis
Properties of the Daubechies Wavelet Packets}
\author[M.\ Nielsen]{Morten Nielsen }
\thanks{This work was supported by the Independent Research Fund Denmark, grant no.\ 5281-00046B}
\address{Department of Mathematical Sciences\\ Aalborg
  University\\ Thomas Manns Vej 23\\ DK-9220 Aalborg East\\ Denmark}
\email{mnielsen@math.aau.dk}

\begin{document}
\begin{abstract}
Nielsen and Zhou (\emph{Mean size of wavelet packets}, ACHA 13 (2002), 22--34)
conjectured that for every Daubechies filter of length at least four the
associated wavelet packets fail to be a Schauder basis of $L^p(\R)$ for every
$p\neq2$; their own $\ell^1/\ell^\infty$ estimate only reaches extreme
exponents. We prove the Schauder-basis-failure half of the conjecture in full,
for every $1<p<\infty$ with $p\neq2$, in the length-four case. The proof
combines convexity of the wavelet packet pressure function with the uniqueness
of equilibrium states for irreducible matrix families due to Feng and
K\"aenm\"aki (2011) and an exact algebraic separation of two periodic
spectral growth rates of the high-pass transition matrices. 
\end{abstract}

\maketitle

\section{Introduction}

Let $a=\{a(k)\}_{k\in\Z}$ be a finitely supported refinement mask satisfying
$\sum_k a(k)=2$ and $\sum_k a(k)a(k+2j)=2\delta_{j,0}$, and let $\phi$ be the
associated scaling function, the compactly supported solution of
$\phi(x)=\sum_{k}a(k)\phi(2x-k)$ with orthonormal integer shifts. Writing
$b(k)=(-1)^ka(1-k)$ for the high-pass mask and $\psi(x)=\sum_k b(k)\phi(2x-k)$
for the wavelet, $\{2^{j/2}\psi(2^jx-k)\}_{j,k\in\Z}$ is an orthonormal
wavelet basis of $L^2(\R)$.

Coifman, Meyer and Wickerhauser~\cite{CMW} refined this construction by
iterating \emph{both} filters at every scale. Set $w_0=\phi$, $w_1=\psi$, and
\begin{equation}\label{eq:packet-recursion}
        w_{2n}(x)=\sum_{k}a(k)w_n(2x-k),
        \qquad
        w_{2n+1}(x)=\sum_{k}b(k)w_n(2x-k),
        \qquad n\in\Z_{\ge0}.
\end{equation}
The functions $\{w_n:n\in\Z_{\ge0}\}$ are the \emph{wavelet packets}
associated with $a$, and their shifts again form an orthonormal basis of
$L^2(\R)$, the \emph{orthonormal wavelet packet basis}. The branch
$\{w_{2^n-1}\}_{n\ge1}$, obtained by iterating only the high-pass filter in
\eqref{eq:packet-recursion}, is exactly the object controlled by the matrices
studied below; S\'er\'e~\cite{Sere} showed it is generically the
worst-localized branch of the tree.

Since the wavelet packets form an orthonormal basis of $L^2(\R)$, it is
natural to ask whether they also form a Schauder basis of $L^p(\R)$ for
$p\neq2$. Paley~\cite{Paley} showed that for the Haar wavelet the wavelet
packets constitute the Walsh system, a Schauder basis of $L^p(\R)$ for every
$1<p<\infty$. In sharp contrast, Coifman, Meyer and Wickerhauser~\cite{CMW},
sharpened by Fan~\cite{Fan}, showed the Meyer wavelet packets are not
uniformly bounded in $L^p$ for large $p$, and Nielsen~\cite{NielsenRMI}
proved the analogous failure of the Schauder-basis property for large $p$ for
Daubechies, least-asymmetric, and Coiflet wavelet packets. All of these
negative results rest on estimating the growth of $L^p$ norms along the
high-pass branch of the wavelet packet tree.

\subsection{The mean size of a subdivision tree and the joint spectral radius}

Nielsen and Zhou~\cite{NielsenZhou} organized this circle of ideas into a
single quantitative framework, the \emph{mean size} of a subdivision tree.
Given a finite family $M=\{a_\varepsilon:\varepsilon\in E\}$ of refinement
masks and a function $\psi$, the subdivision tree $T(M,\psi)$ consists of the
functions $\psi_{\varepsilon_1,\dots,\varepsilon_n}$, $n\ge0$,
$\varepsilon_1,\dots,\varepsilon_n\in E$, defined recursively by
$\psi_{\varepsilon_1,\dots,\varepsilon_n}(x)=\sum_{\alpha}a_{\varepsilon_1}(\alpha)
\psi_{\varepsilon_2,\dots,\varepsilon_n}(2x-\alpha)$, and the mean $L^p$ size
of the tree on scale $n$ is
\begin{equation*}
        M_p(M,\psi):=\lim_{n\to\infty}
        \Bigl(\sum_{\varepsilon_1,\dots,\varepsilon_n}
        \norm{\psi_{\varepsilon_1,\dots,\varepsilon_n}}_p^p\Bigr)^{1/np}.
\end{equation*}
The central result of \cite{NielsenZhou} identifies this limit with the
\emph{$p$-norm joint spectral radius}~\cite{RotaStrang,DL,Wang,Jia} of a
finite family of transition matrices built from $M$: for $\varepsilon\in E$
one sets
\begin{equation}\label{eq:CepsDef}
        C^\varepsilon
        =\bigl\{\bigl(a_\varepsilon(\eta+2\alpha-\beta)\bigr)_{\alpha,\beta}
        :\eta\in\{0,1\}\bigr\},
\end{equation}
and Nielsen and Zhou's Theorem~2 and Corollary~3.1 give, for the branch
obtained by repeating a single symbol $\varepsilon$,
\begin{equation}\label{eq:branch-mean-size}
        M_p^\varepsilon(M,\psi)
        :=\lim_{n\to\infty}
        \norm{\psi_{\varepsilon,\dots,\varepsilon}}_p^{1/n}
        =2^{-1/p}\rho_p(C^\varepsilon).
\end{equation}

\subsection{The Schauder-basis necessary condition and the conjecture}

Using \eqref{eq:branch-mean-size} together with a duality argument for
Schauder bases, Nielsen and Zhou~\cite[Lemma~3.1]{NielsenZhou} proved the
following necessary condition.

\begin{lemma}\label{lem:NZ-necessary}
Suppose the collection of wavelet packets extracted from the branch
$T^\varepsilon(M,\psi)$ is a subset of an orthonormal basis $\mathbb B$ of
$L^2(\R)$ with dense span in $L^p(\R)$, $1\le p<\infty$, satisfying the
uniform lower norm bound \eqref{eq:branch-mean-size} requires. Then a
necessary condition for $\mathbb B$ to be a Schauder basis of $L^p(\R)$ is
\begin{equation*}
        \rho_p(C^\varepsilon)\rho_q(C^\varepsilon)=2,
        \qquad \frac1p+\frac1q=1.
\end{equation*}
\end{lemma}

Specializing to $\varepsilon=1$, the high-pass branch, gives
\cite[Remark~3.3]{NielsenZhou}: for the orthonormal wavelet packets
$\{w_n\}_{n\ge0}$ built from a low-/high-pass filter pair with dense span in
$L^p(\R)$, a necessary condition for $\{w_n\}_{n\ge0}$ to be a Schauder basis
of $L^p(\R)$ is
\begin{equation}\label{eq:NZ-necessary-1}
        \rho_p(C_1)\rho_{p'}(C_1)=2,
        \qquad \frac1p+\frac1{p'}=1,
\end{equation}
where we abbreviate $C_1:=C^1$.

Nielsen and Zhou tested \eqref{eq:NZ-necessary-1} on the one-parameter family
of length-four quadrature mirror filters, using an $\ell^1/\ell^\infty$
estimate of $\rho_1(C_1)\rho_\infty(C_1)$ together with the log-convexity of
$p\mapsto\rho_p(C_1)$. This shows \eqref{eq:NZ-necessary-1} fails, and hence
that the wavelet packets are not a Schauder basis of $L^p(\R)$, once $p$ is
close enough to $1$ or to $\infty$; for the Daubechies length-four filter,
this covers a numerically determined range of extreme exponents. What happens
for $p$ near $2$ was left open, and recorded as a conjecture:

\begin{conjecture}[Nielsen--Zhou~{\cite[\S3.2]{NielsenZhou}}]\label{conj:NZ}
The basic wavelet packets associated with a Daubechies filter of length at
least $4$ will fail to be a Schauder basis for $L^p$ when $p\neq2$, and the
functions will not be uniformly bounded in $p$-mean across scales for any
$p>2$.
\end{conjecture}

\subsection{Contribution of this note}

We prove the first assertion of Conjecture~\ref{conj:NZ} in full, for every
$1<p<\infty$ with $p\neq2$, in the length-four case: the inequality
$\rho_p(C_1)\rho_{p'}(C_1)>2$ of Lemma~\ref{lem:NZ-necessary} holds strictly for
\emph{every} $p\neq2$, not merely near the endpoints. This requires a global
statement about the pressure function
\begin{equation}\label{eq:pressure-intro}
        P(s):=\lim_{n\to\infty}\frac1n\log
        \sum_{\omega\in\{0,1\}^n}\norm{B_\omega}^s,
        \qquad B_\omega:=B_{\omega_n}\cdots B_{\omega_1},
\end{equation}
of the pair $\{B_0,B_1\}=C_1$ of high-pass transition matrices --- that $P$
has \emph{no affine segment} on $(0,\infty)$ --- which a short convexity
argument in Section~\ref{sec:setup} then upgrades to a strict inequality on
the whole range $p\neq2$. This global statement is obtained from the
uniqueness of equilibrium states for irreducible matrix families due to Feng
and K\"aenm\"aki~\cite{FengKaenmaki}, 
combined with an exactly verifiable separation of two periodic spectral
growth rates of $B_0,B_1$, see Sections~\ref{sec:FK} and~\ref{sec:D4}.  

\section{The pressure of the high-pass branch}\label{sec:setup}

By \eqref{eq:branch-mean-size} with $\varepsilon=1$, the pressure $P$
\eqref{eq:pressure-intro} of $C_1=\{B_0,B_1\}$, where $B_0$ and $B_1$ are
constructed explicitly in Section~\ref{sec:D4}, satisfies
\begin{equation}\label{eq:rho-p-pressure}
        \log\rho_p(C_1)=\frac{P(p)}p,
        \qquad 1<p<\infty,
\end{equation}
using $\rho_p(C_1)=\lim_n(\sum_\omega\norm{B_\omega}^p)^{1/np}$, so
\eqref{eq:NZ-necessary-1} is a statement about $P$ alone. Likewise at $p=2$,
\eqref{eq:branch-mean-size} gives $M_2^1(M,\psi)=2^{-1/2}\rho_2(C_1)$
\cite[Cor.~3.1]{NielsenZhou}; since each packet $\psi_{1,\dots,1}$ on the
high-pass branch has $L^2$ norm one, $M_2^1(M,\psi)=1$, so
\begin{equation}\label{eq:P2-normalization}
        \rho_2(C_1)=\sqrt2,
        \qquad\text{equivalently}\qquad P(2)=\log2.
\end{equation}

The goal is to show \eqref{eq:NZ-necessary-1} fails strictly whenever
$p\neq2$. This reduces to two facts: (i) $P$ has no affine segment on
$(0,\infty)$, proved in Sections~\ref{sec:FK} and~\ref{sec:D4}; and (ii)
$P(2)=\log2$, which is \eqref{eq:P2-normalization}. Granting both, set
$F(u):=uP(1/u)$ for $0<u<1$. Notice that for each $n$, $\frac1n\log\sum_{\omega}\|B_\omega\|^s$ is a log-sum-exp of the linear functions $s\mapsto s\log\|B_\omega\|$, hence convex in $s$; since $P(s)$ is the limit of these convex functions, $P$ is itself convex.  
With $u=1/p$,
\begin{equation*}
        \log\Bigl(\frac{\rho_p(C_1)\rho_{p'}(C_1)}2\Bigr)
        =F(1/p)+F(1-1/p)-\log2,
\end{equation*}
and since $u\cdot(1/u)+(1-u)\cdot(1/(1-u))=2$, gives
\begin{equation*}
        F(u)+F(1-u)
        =uP(1/u)+(1-u)P(1/(1-u))
        \geq P(2)=\log2,
\end{equation*}
with equality only if $P$ is affine on the interval with endpoints $1/u$ and
$1/(1-u)$. If $P$ has no affine segment, equality thus occurs only at
$u=1/2$, i.e.\ $p=2$, so
\begin{equation*}
        \rho_p(C_1)\rho_{p'}(C_1)>2,
        \qquad 1<p<\infty,\quad p\neq2,
\end{equation*}
and by \eqref{eq:NZ-necessary-1} the wavelet packet system cannot be a
Schauder basis of $L^p(\R)$ for $p\neq2$. To establish (i), we combine
Feng--K\"aenm\"aki's uniqueness theorem with an exact
spectral-radius separation obtained in Section~\ref{sec:D4}.

\subsection{A Theorem by Feng--K\"aenm\"aki}\label{sec:FK}

Let $M_d(\R)$ denote the vector space of $d\times d$ real matrices, and let
$\A=\{A_1,\dots,A_\ell\}\subset M_d(\R)$. For a word $I=i_1\cdots i_n$, write
$A_I=A_{i_n}\cdots A_{i_1}$. Define
\begin{equation}\label{eq:general-pressure}
        P_\A(q):=\lim_{n\to\infty}\frac1n
        \log\sum_{I\in\{1,\dots,\ell\}^n}\norm{A_I}^{q},
        \qquad q>0.
\end{equation}
The limit exists by subadditivity. The family $\A$ is \emph{irreducible} if
there is no nonzero proper linear subspace of $\R^d$ invariant under every
$A_i$; this is the same notion of irreducibility used throughout the
matrix-pressure literature, e.g.\ \cite[Def.~1.1]{FengKaenmaki}.

We use the following theorem of Feng and K\"aenm\"aki~\cite[Prop.~1.2]
{FengKaenmaki}, stated in precisely the norm-pressure convention
\eqref{eq:general-pressure}. Let $\Sigma=\{1,\ldots,\ell\}^{\N}$, and let
$\mathcal M_\sigma(\Sigma)$ denote the shift-invariant Borel probability
measures on $\Sigma$. The variational principle has the form
\begin{equation}\label{eq:variational}
        P_\A(q)=\sup_{\mu\in\mathcal M_\sigma(\Sigma)}
        \bigl(h(\mu)+q\Lambda(\mu)\bigr),
\end{equation}
where $h(\mu)$ is entropy and $\Lambda(\mu)$ is the top Lyapunov exponent
associated with the matrix cocycle. A measure attaining the supremum in
\eqref{eq:variational} is called a \emph{$q$-equilibrium state}.

\begin{theorem}[{\cite[Prop.~1.2]{FengKaenmaki}}]\label{thm:FK}
Let $\A=\{A_1,\dots,A_\ell\}\subset M_d(\R)$ be irreducible. Then for every
$q>0$ the norm pressure $P_\A(q)$ in \eqref{eq:general-pressure} has a
unique $q$-equilibrium state $\mu_q$. Moreover, $\mu_q$ satisfies the Gibbs
estimate: there exists $C_q\geq1$ such that for every word $I$ of length $n$,
\begin{equation}\label{eq:gibbs}
        C_q^{-1}e^{-nP_\A(q)}\norm{A_I}^{q}
        \leq \mu_q([I])\leq
        C_q e^{-nP_\A(q)}\norm{A_I}^{q}.
\end{equation}
Furthermore $P_\A$ is differentiable on $(0,\infty)$ and
\begin{equation}\label{eq:FK-derivative}
        P_\A'(q)=\Lambda(\mu_q),\qquad q>0.
\end{equation}
\end{theorem}

For a general, not necessarily irreducible, family, Feng and K\"aenm\"aki
show only that there are at most $d$ ergodic $q$-equilibrium states
\cite[Thm.~1.7]{FengKaenmaki}. However, the irreducibility of
$\{B_0,B_1\}$, established in Proposition~\ref{prop:irreducible} below,
provides the uniqueness needed here.

The following consequence is the central thermodynamic observation of this
note: on any interval where the pressure is affine, \emph{every} word,
periodic or not, must grow at the same normalized rate.

\begin{proposition}
\label{prop:affine-rigidity}
Let \(A=\{A_1,\ldots,A_\ell\}\subset M_d(\mathbb R)\) be an irreducible family of
invertible matrices. Suppose that \(P_A\) is affine on a nondegenerate interval
\([q_0,q_1]\subset(0,\infty)\), say
\[
        P_A(q)=a+bq,\qquad q\in[q_0,q_1].
\]
Then every word \(I\) satisfies
\[
        \rho(A_I)^{1/|I|}=e^b .
\]
\end{proposition}

\begin{proof}
Choose \(q_0<r<s<q_1\). Let \(\mu_r\) be the unique \(r\)-equilibrium state.
Since \(P_A(q)=a+bq\) on \([q_0,q_1]\), its derivative at \(r\) is \(b\).
By (2.6),
\[
        \Lambda(\mu_r)=P_A'(r)=b.
\]
Because \(\mu_r\) is an \(r\)-equilibrium state,
\[
        h(\mu_r)+r\Lambda(\mu_r)=P_A(r)=a+br,
\]
and hence \(h(\mu_r)=a\). Therefore, for every \(q\in[q_0,q_1]\),
\[
        h(\mu_r)+q\Lambda(\mu_r)=a+qb=P_A(q).
\]
Thus \(\mu_r\) is a \(q\)-equilibrium state throughout the interval. By uniqueness,
all equilibrium states on the interval coincide; in particular \(\mu_r=\mu_s=:\mu\).

Apply the Gibbs estimate \eqref{eq:gibbs} at the two exponents $r$ and $s$ to
a word $I$ of length $n$. Since the measures $\mu=\mu_r=\mu_s$ are the same,
and since $A_I$ is nonsingular, dividing the two inequalities gives a constant
$K\geq1$, independent of $I$ and $n$, such that
\begin{equation*}
        K^{-1}e^{n(P_\A(s)-P_\A(r))}
        \leq \norm{A_I}^{s-r}
        \leq K e^{n(P_\A(s)-P_\A(r))}.
\end{equation*}
Because $P_\A(s)-P_\A(r)=b(s-r)$, this becomes
\begin{equation}\label{eq:word-norm-uniform}
        K^{-1/(s-r)}e^{nb}
        \leq \norm{A_I}
        \leq K^{1/(s-r)}e^{nb}.
\end{equation}
Now fix a word $I$  and apply \eqref{eq:word-norm-uniform} to
$I^m$, the $m$-fold repetition of $I$. Since $A_{I^m}=A_I^m\not=0$ and
$|I^m|=m|I|$, we obtain
\begin{equation*}
        K^{-1/(s-r)}e^{m|I|b}
        \leq \norm{A_I^m}
        \leq K^{1/(s-r)}e^{m|I|b}.
\end{equation*}
Taking the power $1/(m|I|)$ and letting $m\to\infty$, the Gelfand formula
gives
\begin{equation*}
        \spr(A_I)^{1/|I|}=e^b.
\end{equation*}
This proves the claim.
\end{proof}

Consequently, to prove that $P_\A$ has \textit{no affine segment}, it is enough to
find two words $I,J$ with different normalized spectral radii:
\begin{equation*}
        \spr(A_I)^{1/|I|}\neq \spr(A_J)^{1/|J|}.
\end{equation*}

\section{The Daubechies length-four matrices and  affine
segments}\label{sec:D4}

For the length-four Daubechies low-pass filter~\cite[Ch.~6]{Daubechies},
\begin{equation*}
        a_0(0)=\frac{1+\sqrt3}{4},\quad
        a_0(1)=\frac{3+\sqrt3}{4},\quad
        a_0(2)=\frac{3-\sqrt3}{4},\quad
        a_0(3)=\frac{1-\sqrt3}{4}.
\end{equation*}
This is the endpoint $\theta=(1+\sqrt3)/4$ of the one-parameter family of
length-four quadrature mirror filters considered in
\cite[Example~2]{NielsenZhou}. 
Following the shifted high-pass convention used in Nielsen--Zhou
\cite[Example 2]{NielsenZhou} for the subdivision matrices entering \(C_1\),
we define
\[
        a_1(k)=(-1)^k a_0(k).
\]
so
\begin{equation*}
        a_1(0)=\frac{1+\sqrt3}{4},\quad
        a_1(1)=-\frac{3+\sqrt3}{4},\quad
        a_1(2)=\frac{3-\sqrt3}{4},\quad
        a_1(3)=\frac{\sqrt3-1}{4}.
\end{equation*}
With the convention \eqref{eq:CepsDef} for $\varepsilon=1$,
\begin{equation*}
        (B_\eta)_{\alpha\beta}=a_1(\eta+2\alpha-\beta),
        \qquad 0\leq\alpha,\beta\leq2,
\end{equation*}
where $a_1(k)=0$ outside $\{0,1,2,3\}$, the two matrices $C_1=\{B_0,B_1\}$
are
\begin{equation*}
B_0=
\begin{pmatrix}
\frac{1+\sqrt3}{4}&0&0\\[3pt]
\frac{3-\sqrt3}{4}&-\frac{3+\sqrt3}{4}&\frac{1+\sqrt3}{4}\\[3pt]
0&\frac{\sqrt3-1}{4}&\frac{3-\sqrt3}{4}
\end{pmatrix},
\end{equation*}
\begin{equation*}
B_1=
\begin{pmatrix}
-\frac{3+\sqrt3}{4}&\frac{1+\sqrt3}{4}&0\\[3pt]
\frac{\sqrt3-1}{4}&\frac{3-\sqrt3}{4}&-\frac{3+\sqrt3}{4}\\[3pt]
0&0&\frac{\sqrt3-1}{4}
\end{pmatrix}.
\end{equation*}
They are both invertible:
\begin{equation*}
        \det B_0=-\frac{1+\sqrt3}{8},
        \qquad
        \det B_1=\frac{1-\sqrt3}{8}.
\end{equation*}
In particular $\{B_0,B_1\}$ is a non-trivial family in the sense of
\cite{FengKaenmaki}, so the pressure $P$ never takes the value $-\infty$ and
Theorem~\ref{thm:FK} applies once irreducibility is verified. We now verify
that the pair $\{B_0,B_1\}$ is irreducible.

\begin{proposition}\label{prop:irreducible}
The pair $\{B_0,B_1\}$ is irreducible over $\R^3$.
\end{proposition}

\begin{proof}
It is enough to show that the unital algebra generated by $B_0$ and $B_1$ is
all of $M_3(\R)$. Indeed, if a nonzero proper subspace were invariant under
both matrices, then it would be invariant under every element of the
generated algebra, which is impossible for the full matrix algebra.

Vectorize $3\times3$ matrices by stacking their columns, and set
\begin{equation*}
\mathcal B=\bigl(I,B_0,B_1,B_0^2,B_0B_1,B_1B_0,B_1^2,B_0^2B_1,B_0B_1B_0\bigr).
\end{equation*}
A direct computation gives
\begin{equation*}
\det\bigl[\operatorname{vec}(M):M\in\mathcal B\bigr]
=\frac{3(15-7\sqrt3)}{4096}\neq0,
\end{equation*}
up to the sign coming from the convention used to order the nine
vectorized matrices. Since this determinant is nonzero, the nine matrices in
$\mathcal B$ are linearly independent. Hence they form a basis for
$M_3(\R)$, and the algebra generated by $B_0,B_1$ is all of $M_3(\R)$. Thus
the pair is irreducible.
\end{proof}

\subsection{Two exact spectral radii}

\begin{proposition}\label{prop:two-rates}
The words $B_0$ and $G:=B_0B_1^2$ have different normalized spectral radii.
More precisely,
\begin{equation}\label{eq:B0-radius}
        \spr(B_0)=\frac{\sqrt3+\sqrt{11}}4>1,
\end{equation}
whereas
\begin{equation}\label{eq:G-radius}
        \spr(G)<1.
\end{equation}
Consequently,
\begin{equation}\label{eq:different-rates}
        \spr(B_0)\neq \spr(B_0B_1^2)^{1/3}.
\end{equation}
\end{proposition}

\begin{proof}
The characteristic polynomial of $B_0$ factors as
\begin{equation*}
        \chi_{B_0}(x)=
        \left(x-\frac{1+\sqrt3}{4}\right)
        \left(x^2+\frac{\sqrt3}{2}x-\frac12\right).
\end{equation*}
The roots of the quadratic factor are $(-\sqrt3+\sqrt{11})/4$ and
$-(\sqrt3+\sqrt{11})/4$, and the latter has the largest modulus among the
three eigenvalues of $B_0$ and \eqref{eq:B0-radius} follows. 

Now set $G=B_0B_1^2$. A direct calculation gives
\begin{equation}\label{eq:char-G}
\chi_G(x)
=x^3-\frac{7+5\sqrt3}{32}x^2
+\frac{-91+9\sqrt3}{256}x
+\frac{\sqrt3-1}{256}.
\end{equation}
Let $q(x)$ denote the right-hand side of \eqref{eq:char-G}. We record the
following exact signs:
\begin{align*}
q(-1/2)&=-\frac{3(1+9\sqrt3)}{512}<0,
& q(-1/4)&=\frac{3(19-5\sqrt3)}{1024}>0,\\
q(0)&=\frac{\sqrt3-1}{256}>0,&q(1/50)&=\frac{3(-29587+12125\sqrt3)}{8000000}<0,\\
q(3/4)&=\frac{29-59\sqrt3}{1024}<0,
& q(1)&=\frac{3(18-5\sqrt3)}{128}>0.
\end{align*}
Therefore $q$ has one real zero in each of the intervals
\begin{equation*}
        \left(-\frac12,-\frac14\right),
        \qquad
        \left(0,\frac1{50}\right),
        \qquad
        \left(\frac34,1\right).
\end{equation*}
Since $q$ is cubic, these are \textit{all the zeros}. All eigenvalues of $G$
therefore have modulus strictly less than $1$, and \eqref{eq:G-radius}
follows. Combining \eqref{eq:B0-radius} and \eqref{eq:G-radius} gives
\eqref{eq:different-rates}.
\end{proof}

This in turn implies that the pressure has no affine segment, as recorded in
the following lemma.

\begin{lemma}\label{thm:no-affine}
Let $P$ be the pressure \eqref{eq:pressure-intro} of the Daubechies
length-four high-pass matrix pair $\{B_0,B_1\}$. Then $P$ has no affine
segment on $(0,\infty)$.
\end{lemma}

\begin{proof}
By Proposition~\ref{prop:irreducible}, the pair $\{B_0,B_1\}$ is
irreducible. Therefore Theorem~\ref{thm:FK} applies.

Suppose, to obtain a contradiction, that $P$ is affine on a nondegenerate
interval. Proposition~\ref{prop:affine-rigidity} then implies that every
word has the same normalized spectral radius. In particular,
\begin{equation*}
        \spr(B_0)=\spr(B_0B_1^2)^{1/3}.
\end{equation*}
This contradicts Proposition~\ref{prop:two-rates}. Hence $P$ has no affine
segment.
\end{proof}

\section{The main result}\label{sec:main}
We can now prove the main result of this note.

\begin{theorem}\label{thm:main}
For the basic Daubechies length-four wavelet packets, the Schauder-basis
necessary condition \eqref{eq:NZ-necessary-1} fails strictly: for every
$1<p<\infty$ with $p\neq2$,
\begin{equation}\label{eq:main-strict}
        \rho_p(C_1)\rho_{p'}(C_1)>2.
\end{equation}
Consequently, the Daubechies length-four basic wavelet packets do not form a
Schauder basis of $L^p(\R)$ for any $1<p<\infty$, $p\neq2$.
\end{theorem}

\begin{proof}
Let $u=1/p$. Then $1-u=1/p'$ and $u\in(0,1)$. By
\eqref{eq:rho-p-pressure},
\begin{equation*}
        \log\bigl(\rho_p(C_1)\rho_{p'}(C_1)\bigr)
        =uP(1/u)+(1-u)P(1/(1-u)).
\end{equation*}
By convexity of $P$,
\begin{equation*}
        uP(1/u)+(1-u)P(1/(1-u))
        \geq P\bigl(u(1/u)+(1-u)(1/(1-u))\bigr)
        =P(2)=\log2.
\end{equation*}
If $p\neq2$, then $u\neq1/2$, so the two points $1/u$ and $1/(1-u)$ are
distinct. Equality in the convexity inequality would force $P$ to be affine
on the interval between these two points. This is impossible by
Lemma~\ref{thm:no-affine}. Therefore the inequality is strict:
\begin{equation*}
        \log\bigl(\rho_p(C_1)\rho_{p'}(C_1)\bigr)>\log2,
\end{equation*}
which is exactly \eqref{eq:main-strict}. The conclusion follows from the
Nielsen--Zhou necessary condition \eqref{eq:NZ-necessary-1}.
\end{proof}

We conclude this note by the following two remarks.
\begin{remark}
Theorem~\ref{thm:main} settles the Schauder-basis-failure clause of
Conjecture~\ref{conj:NZ} for the length-four filter, for the full range
$p\neq2$. The second clause of the conjecture, that the wavelet packets fail
to be uniformly bounded in $p$-mean across scales for every $p>2$, is a
different (and a priori stronger, quantitative) statement about the growth
rate $M_p^1(M,\psi)=2^{-1/p}\rho_p(C_1)$ itself rather than about the product
$\rho_p(C_1)\rho_{p'}(C_1)$, and is not addressed by the present method. To
the best of our knowledge, it remains open.
\end{remark}
\begin{remark}
The qualitative mechanism of Sections~\ref{sec:setup} and~\ref{sec:D4}
carries over unchanged to wavelet filters of arbitrary finite length. However,
the associated computational complexity increases significantly for longer
filters. Often only numerical approximations of the filter coefficients are
available, which adds to the difficulty of reaching a definitive conclusion.
That said, we have performed numerical calculations for Daubechies filters of
length $6$. These computations provide strong evidence for the same conclusion
as in the length-four case, although we do not include a rigorous
interval-arithmetic verification here.

\end{remark}

\begin{small}

\end{small}

\end{document}